\documentclass[12pt]{amsart}

\usepackage{amsfonts,amssymb, amscd, latexsym, graphicx, psfrag, color}
\usepackage[all]{xy}

\newtheorem{dummy}{dummy}[section]

\newtheorem{theorem}[dummy]{Theorem}

\newtheorem{proposition}[dummy]{Proposition}
\theoremstyle{definition}
\newtheorem{definition}[dummy]{Definition}
\newtheorem{example}[dummy]{Example}
\newtheorem{remark}[dummy]{Remark}

\newcommand{\cC}{\mathcal{C}}
\newcommand{\cL}{\mathcal{L}}

\newcommand{\bC}{\mathbb{C}}
\newcommand{\bD}{\mathbf{D}}

\newcommand{\bQ}{\mathbb{Q}}
\newcommand{\bR}{\mathbb{R}}

\newcommand{\Fun}{\mathrm{Fun}}
\newcommand{\alg}{\mathrm{alg}}

\newcommand{\CC}{\mathit{CC}}
\newcommand{\BM}{\mathrm{BM}}
\newcommand{\RGa}{\mathbf{R}\Gamma}
\newcommand{\Vect}{\mathit{Vect}}

\newcommand{\fj}{\mathfrak{j}}

\begin{document}

\title{Constructible functions and Lagrangian cycles on orbifolds}
\author{Davesh Maulik and David Treumann}
\maketitle

\begin{abstract}
Given a smooth manifold $M$, Kashiwara's index formula expresses the weighted Euler characteristic of a constructible function in terms of its characteristic cycle.  In this note, we generalize this formula to the case when $M$ is a smooth orbifold, answering a question of Behrend.
\end{abstract}

\tableofcontents

\section{Introduction}

Let $M$ be a smooth compact manifold with cotangent bundle $T^*M$, and let $\zeta_M \subset T^*M$ denote the zero section.  It is a classical result that the homological intersection number $\zeta_M \cap \zeta_M$ is equal to the Euler characteristic of $M$.  The ``index formula'' \cite{BDK,K} is a generalization to ``weighted'' Euler characteristics.

\begin{theorem}[Kashiwara index formula]
\label{thm:intro1}
Let $M$ be a compact smooth manifold.  Let $f$ be a constructible function on $M$, and let $\CC(f)$ be its characteristic cycle in $T^*M$.  Let $\zeta_M$ denote the zero section in $T^* M$.  The intersection number $\CC(f) \cap \zeta_M$ and the weighted Euler characteristic $\int f$ coincide:
\[
\int f = \CC(f) \cap \zeta_M
\]
\end{theorem}

The classical result is the special case when $f$ is identically 1.  We will recall the definitions later.  

The purpose of this note is to generalize Theorem \ref{thm:intro1} to
the case when $M$ is a smooth orbifold.  The problem of doing so arose
in Behrend's work \cite{Behrend} on Deligne-Mumford stacks
equipped with a symmetric obstruction theory.    For a
large class of such stacks, he used Theorem \ref{thm:intro1} to
express the degree of the virtual fundamental class in terms of the
weighted Euler characteristic of a certain constructible function (see
Section \ref{refinedbehrend}).  This ``Behrend function'' has been
central in recent progress on Donaldson-Thomas theory of Calabi-Yau
threefolds \cite{Br,JS}.

Behrend already observed in loc. cit. that the generalization of
Theorem \ref{thm:intro1} to orbifolds would allow one to remove most
of the restrictions on $X$.  We expect this to have applications in
Donaldon-Thomas theory for more general geometries, where nontrivial orbifold structure arises
naturally, e.g. when studying sheaves on threefolds relative to divisors.

The problem of generalizing Theorem \ref{thm:intro1} at least has the potential to be somewhat subtle.  For instance, when $M$ is an orbifold the homological intersection $\zeta_M \cap \zeta_M$ is usually a rational number, not an integer.  In particular it can't be a finite alternating sum of betti numbers, so one has to be careful about what's meant by Euler characteristic.  Moreover the study of Morse theory and deeper ``quantum'' questions about Lagrangian intersections are expected to give rise to new phenomena in the orbifold case.  

Nevertheless in this paper we show that a suitable orbifold version of the index formula is essentially a formal consequence of the usual index formula.  To do this requires us to study the local-to-global behavior of constructible functions and Lagrangian cycles.  The fact that both quantities in the formula are ill-defined for non-compactly supported functions is a small obstacle to doing this.  To overcome this obstacle we use the fact that constructible functions with compact support, and Lagrangian cycles with compact ``horizontal'' support, form \emph{cosheaves}.

\subsection{Notation}

$\Vect$ denotes the abelian category of rational vector spaces.  We write $T^*M$ for the cotangent bundle of $M$.  When $N$ is a submanifold of $M$ we write $T^*_N M$ for the conormal bundle to $N$ in $M$.  We write $H_i, H_i^{\BM}, H^i$ and $H^i_c$ for the functors of respectively homology, Borel-Moore homology, cohomology, and compactly suported cohomology, which we consider with rational coefficients.

\subsection*{Acknowledgements}

We thank Kai Behrend, Cheol-Hyun Cho and Pierre Schapira for comments on an earlier draft.  We are grateful to Dima Tamarkin for teaching us about cosheaves.  D.M. is partially supported by a Clay Research Fellowship.

\section{Constructible functions and Lagrangian cycles on manifolds}

In this section we review the theory of constructible functions and characteristic cycles.  We mostly follow \cite[Chapter 9]{KS}.

We will consider subanalytic Whitney stratifications of real analytic manifolds $M$, which form a partially ordered set under refinement.  An essential result is that this poset if filtered: any two subanalytic Whitney stratifications have a common refinement (\cite[Theorem 8.3.20]{KS}).  As a consequence if $G$ is a finite group acting on $M$ then any stratification of $M$ can be refined to a $G$-invariant stratification.


If $S$ is a Whitney stratification, we write $\Lambda_S \subset T^*M$ for its characteristic variety, thus
$$\Lambda_S = \bigcup_{N \in S} T_N^*M$$
where $N$ runs through the strata of $S$.

\subsection{Constructible functions and Euler integration}

Let $S$ be a subanalytic Whitney stratification of $M$.  We will call a compactly supported function $f:M \to \bQ$ $S$-constructible if it is constant along the strata of $S$.  Write $\Fun_c(M,S)$ for the group of compactly supported $S$-constructible functions, and
$$\Fun_c(M) = \bigcup_S \Fun_c(M,S)$$
We define an operator $\int:\Fun_c(M) \to \bQ$ as follows:
\begin{definition}
For $f \in \Fun_c(M)$, we set
$$\int_M f = \sum_{t\in \bQ} \chi_c(f^{-1}(t)) \cdot t$$
Note that as $f$ is constructible this sum is finite.
\end{definition}

\subsection{Lagrangian cycles and intersection}
\label{subsec:lcai}

For each stratification $S$ we define the group of \emph{conic Lagrangian cycles} supported on $\Lambda_S$ to be $H_n^{\BM}(\Lambda_S)$.   As $\Lambda_S$ is $n$-dimensional, a member of this group is given by choosing for each component of $\Lambda_S^\circ \subset \Lambda_S$ (the smooth points of $\Lambda_S$) an orientation and multiplicity, satisfying a linear condition.  We define the ``support'' of the cycle to be the closure of the components with nonzero multiplicities.

Write $\cL(M)$ for the group of Lagrangian cycles on $T^*M$ that are supported on the conormal variety of \emph{some} subanalytic Whitney stratification :
$$\cL(M) = \bigcup_S H_n^\BM(\Lambda_S)$$
We will say that a Lagrangian cycle has \emph{compact horizontal support} if the projection of its support to $M$ is compact.  These form a subgroup $\cL_c(M) \subset \cL(M)$

Let $u:M \to \bR$ be a smooth function (it is not necessary for it to be subanalytic) and let $\Gamma \subset T^*M$ be the graph of its derivative, i.e.
$$\Gamma = \{(x,\xi) \in T^*M \mid \xi = du_x\}$$
Let us pick an orientation of $M$.  If $\Gamma \cap \Lambda_S = \Gamma \cap \Lambda^\circ_S$, these intersection points are isolated, and the intersection is transverse, then we can define a local intersection number $(\eta \cap \Gamma)_x \in \bQ$ for any $\eta$ supported on $\Lambda_S$.  If furthermore $\eta$ has compact horizontal support then the sum
$$\sum_{x \in \Lambda_S \cap \Gamma} (\eta \cap \Gamma)_x$$
is finite and independent of the $u$ chosen.  We denote the map $\cL_c(M) \to \bQ$ by $\cap \zeta_M$, where $\zeta_M$ denotes the zero section of $T^*M$ endowed with an orientation.

\subsection{The characteristic cycle of a constructible function}

If $S$ is a stratification of $M$ and $f:M \to \bQ$ is a $S$-constructible function, then we may define a element $\CC(f) \in H_n^\BM(\Lambda_S)$, called the \emph{characteristic cycle} of $f$.  See \cite[Chapter 9]{KS} for details.  The support of $\CC(f)$ projects onto the support of $f$, so $\CC$ carries $\Fun_c(M)$ to $\cL_c(M)$.  
\begin{theorem}[{\cite[Theorems 9.7.1, 9.7.10]{KS}}]
\label{thm:charcycleiso}
For each real analytic manifold $M$, the map $\CC:\Fun_c(M) \to \cL_c(M)$ is an isomorphism.
\end{theorem}

\begin{theorem}[Index formula]
For each $M$, the triangle
$$\xymatrix{
\Fun_c(M) \ar[rr]^{\CC} \ar[dr]_{\int} &  &  \cL_c(M) \ar[dl]^{\cap \zeta} \\ 
& \bQ
}
$$
commutes

\end{theorem}

See the remarks at the end of \cite[Chapter 9]{KS} for a history of this result.  In the real setting we are working in, the theorem is due to Kashiwara \cite{K}.

\section{Cosheaves}

It is intuitively plausible that constructible functions, conic Lagrangian cycles, and the index formula have a local nature on $M$.  A typical way to express this locality is using the theory of sheaves, but for the purposes of this paper we have found that the equivalent but less well-known theory of \emph{cosheaves} to be more convenient.  Actually we will not work with cosheaves on a single manifold but on a Grothendieck site containing all manifolds of a fixed dimension.  For simplicity we work with oriented manifolds.

\subsection{The subanalytic Grothendieck site}

Fix a dimension $n$, and let $\cC$ be the category whose objects are real analytic manifolds equipped with an orientation, and whose morphisms are real subanalytic maps that are open immersions.  We endow $\cC$ with the following Grothendieck topology: if $\{U_i \to U\}_{i \in I} \subset \cC_{/U}$ is a sieve in $\cC$, we will say it is an \emph{covering sieve} if there there is a finite list $U_{i_0} \to U,\ldots,U_{i_m} \to U$ so that every point of $U$ belongs to one of the $U_i$.

\begin{remark}
\label{rem:fincond}
This is the same topology considered in \cite{KS2}.  The finiteness condition makes this \emph{different} than the usual topology on real analytic manifolds.  For instance, the collection of intervals $(\frac{1}{n},1)$ do not cover the interval $(0,1)$ in the topology we are considering.  
\end{remark}

By a \emph{pre-cosheaf} on $\cC$ we will mean a covariant functor $F:\cC \to \Vect$.  A pre-cosheaf $F$ is a cosheaf if the natural map
$$\varinjlim_{i \in I} F(U_i) \to F(U)$$
is an isomorphism whenever $\{U_i\}_{i \in I}$ is a covering sieve of $U$.  

\begin{remark}
\label{rem:dereq}
There is a dictionary between sheaves and cosheaves: if $F$ is a sheaf on a site $S$, then the assignment $U \mapsto \RGa_c(U,F)$ is \emph{co}variant for open inclusions, and we may interpret $\RGa_c(-,F)$ as a complex of cosheaves.  When $S$ is a locally compact Hausdorff space, it can be shown that the functor $F \mapsto \RGa_c(-,F)$ induces an equivalence of derived categories.  We do not know whether this is true for $S = \cC$, but we will see that some of the cosheaves we consider arise from sheaves in this way.
\end{remark}

Recall that each presheaf $P$ has a sheafification $P^\dagger$, and there is a natural map $P \to P^\dagger$ that is universal among maps from $P$ to sheaves.  Similarly, each pre-cosheaf has a natural cosheafification $P_\dagger \to P$ that is universal among maps from cosheaves to $P$.  Our main example is the \emph{constant cosheaf}, which is the cosheafification of the constant pre-cosheaf given by
$$P(U) = \bQ$$
We denote the constant cosheaf by $\bD$, because in the derived equivalence of \ref{rem:dereq} it corresponds to the Verdier dualizing complex.

\subsection{The cosheaf of constructible functions}

To each subanalytic inclusion $U \hookrightarrow V$, extension by zero provides an inclusion $\Fun_c(U) \hookrightarrow \Fun_c(V)$.  In other words, $\Fun_c$ is a pre-cosheaf on $\cC$.

\begin{proposition}
\label{prop:funccosheaf}
$\Fun_c$ is a cosheaf on $\cC$.
\end{proposition}

\begin{proof}
Let $I = \{U_i \to M\}$ be a covering sieve of $M$.  We have to show that the natural map
$$c:\varinjlim_I \Fun_c(U_i) \to \Fun_c(M)$$
is an isomorphism.  By the definition of covering sieves in $\cC$, there are finitely many maximal charts $U_{i_1},\ldots U_{i_m}$ in the sieve $I$.  A standard argument reduces us to the case where $m = 2$, in which case the limit coincides with the pushout of the diagram 
$$\Fun_c(U_{i_1}) \leftarrow \Fun_c(U_{i_1} \cap U_{i_2}) \to \Fun_c(U_{i_2})$$
To prove the Proposition we therefore only have to show that if $M = U \cup V$ where $U$ and $V$ are subanalytic open charts, then the sequence
\[
\Fun_c(U \cap V) \stackrel{f \mapsto (f,f)}{\longrightarrow}  \Fun_c(U) \oplus \Fun_c(V) \stackrel{(f,g) \mapsto f - g}{\longrightarrow} \Fun_c(M) \to 0
\]
is exact.

If $f \in \Fun_c(U)$ and $g \in \Fun_c(V)$ and $f - g = 0$ in $\Fun_c(M)$, then $f = g$ and the common support must belong to $U \cap V$.  This shows that the sequence is exact in the middle.  Let us show that it is exact at the right.  Let $f:M \to \bQ$ be a compactly supported function constructible with respect to a subanalytic Whitney stratification $S$.  By refining $S$ if necessary, we may assume that the closure of each stratum of $S$ in the support of $f$ is entirely contained in either $U$ or $V$.  The indicator function of each such stratum belongs to either $\Fun_c(U)$ or $\Fun_c(V)$, and since $f$ is a linear combination of such indicator functions it belongs to the image of the map $\Fun_c(U) \oplus \Fun_c(V) \to \Fun_c(M)$.

\end{proof}

By definition, the map $\int$ is a morphism of pre-cosheaves from $\Fun_c$ to the constant pre-cosheaf.  By the Proposition, $\Fun_c$ is a cosheaf and so this map factors through the 
cosheafification $\bD$ of the constat pre-cosheaf.  We denote the map by $\int:\Fun_c \to \bD$.

\subsection{The cosheaf of Lagrangian cycles}

Each open immersion $U \hookrightarrow V$ induces an open immersion $T^*U \hookrightarrow T^*V$.  If the inclusion is subanalytic then we may extend any Whitney stratification of $U$ to one of $V$, and a horizontally compact conic Lagrangian cycle on $T^*U$ extends by zero to one on $T^*V$.  This gives us an inclusion $\cL_c(U) \hookrightarrow \cL_c(V)$, i.e. $\cL_c$ is a pre-cosheaf on $\cC$.

\begin{proposition}
$\cL_c$ is a cosheaf on $\cC$
\end{proposition}

\begin{proof}
It is possible to prove this directly, but let us simply note that by Theorem \ref{thm:charcycleiso}, there is an isomorphism of pre-cosheaves $\cL_c \cong \Fun_c$, and since $\Fun_c$ is a cosheaf by Proposition \ref{prop:funccosheaf}, $\cL_c$ is also a cosheaf.
\end{proof}

Fix $\eta \in \cL_c(U)$.  The intersection multiplicity $\eta \cap \zeta$ is defined by choosing a suitably generic smooth function $u:U \to \bR$ and adding up local multiplicity of the intersection $\eta \cap \Gamma_{du}$.  As $\eta$ has compact support, we may assume that $u$ vanishes outside of a compact set.  By extending this function by zero further along an open inclusion $U \hookrightarrow V$, we see that the triangle
$$\xymatrix{
\cL_c(U) \ar[rr] \ar[dr]_{\cap \zeta} & & \cL_c(V) \ar[dl]^{\cap \zeta} \\
& \bQ
}
$$
commutes, i.e. we have a map from $\cL_c$ to the constant pre-cosheaf.  This induces a map $\cL_c \to \bD$ that we continue to denote by $\cap \zeta$.

\subsection{The local index formula}

The maps $\CC:\Fun_c(M) \to \cL_c(M)$ clearly assemble to a morphism of cosheaves, we therefore have the ``local'' or ``cosheaf'' version of the index formula

\begin{theorem}
\label{thm:localbdk}
The triangle
$$\xymatrix{
\Fun_c \ar[rr]^{\CC} \ar[dr]_{\int} &  &  \cL_c \ar[dl]^{\cap \zeta} \\ 
& \bD
}
$$
of cosheaves on $\cC$ commutes.
\end{theorem}

\section{Application to orbifolds}

In this section $\cC$ continues to denote the subanalytic Grothendieck site of oriented $n$-manifolds.  If $\cC'$ is any other Grothendieck site equipped with a map to $\cC$, the cosheaves $\bD$, $\Fun_c$, and $\cL_c$ pull back to $\cC'$ and the pullback triangle of Theorem \ref{thm:localbdk} still commutes.  In this section we will study this triangle when $\cC'$ is the Grothendieck topology associated to a real analytic orbifold.

\subsection{The subanalytic site of an orbifold}

There are many possible foundations for a theory of orbifolds, we will use those set up in \cite{BGNX}.  Let $X = [X_1 \rightrightarrows X_0]$ be an orbifold in the sense of loc. cit. Section 1.7.  We let $\overline{X}$ denote the coarse space of $X$ in the sense of loc. cit. Section 1.1.4.  For every topological space $Y$ we have a notion of a map from $Y$ to $X$ (in fact there is a groupoid of such maps), there is a good notion of fiber product of maps over $X$, and it is possible to single out a class of \'etale maps to $X$.

By a oriented, real analytic structure on $X$, we will mean a finite collection of \'etale maps $U_1 \to X,\ldots,U_m \to X$ that cover $\overline{X}$ such that:
\begin{itemize}
\item Each $U_i$ is an $n$-dimensional oriented real analytic manifold
\item For each $i$ and $j$, the map $U_i \times_X U_j \to U_i$ is real analytic and orientation-preserving.
\end{itemize}
If $U$ is another oriented real analytic manifold, we say that a map $U \to X$ is real analytic if the induced map $U \times_X U_i \to U_i$ is real analytic and orientation preserving for $i =1,\ldots, m$.

If $X$ is an oriented real analytic orbifold, define a category $\cC_{/X}$ in the following way:
\begin{enumerate}
\item The objects are real analytic, orientation-preserving \'etale maps $U \to X$
\item The morphisms are orientation preserving open immersions (sic) $V \to U$ over $X$.
\end{enumerate}

\begin{remark}
In (2), we are taking advantage of the fact that the subanalytic topology is fine enough that the ``usual'' topology and ``\'etale'' topologies coincide.
\end{remark}

A sieve $\{U_i \to U\}$ in $\cC_{/X}$ is a covering sieve if and only if it is a covering sieve in $\cC$.  There is thus evidently a continuous functor $\mathfrak{j}_X:\cC_{/X} \to \cC$.  In other words, if $F:\cC \to \Vect$ is a cosheaf on $\cC$, then $F \circ \fj_X$ is a cosheaf on $\cC_{/X}$.  

\begin{remark}
Note that $X$ itself is not always an object of $\cC_{/X}$, however we abuse notation and write $F(X)$ for the ``global sections'' of $F$ over $\cC_{/X}$, i.e. 
$$F(X) = \varinjlim_{(U \to X) \in \cC_{/X}} F(U)$$
Moreover, even if $X$ is not an object of $\cC_{/X}$, we still have a notion of a covering sieve on $X$.  It is a set $S$  of objects of $\cC_{/X}$ with the following properties:
\begin{enumerate}
\item (Sieve property) if $(U \to X) \in S$ and $V$ admits a map to $U$, then $V$ is in $S$ as well
\item (Covering property) there is a finite collection of maps $U_1 \to X, \cdots, U_n \to X$ in $S$ such that the induced maps $U_i \to \overline{X}$ cover $\overline{X}$.
\end{enumerate}
If $S$ is a covering sieve and $F$ is a cosheaf on $\cC_{/X}$ then the natural map to $F(X)$ is an isomorphism:
$$\varinjlim_{(U \to X) \in S} F(U) \stackrel{\sim}{\to} F(X)$$

\end{remark}

\subsection{The orbifold index formula, abstract version}

We make the following definition:

\begin{definition}
\label{def:funcorb}
Let $X$ be an orbifold, let $\cC_{/X}$ be its subanalytic \'etale site, and let $\fj_X$ be the natural forgetful functor $\cC_X \to \cC$.
\begin{enumerate}
\item 
A \emph{constructible function} on $X$ is a section of the cosheaf $\Fun_c \circ \fj_X$ on $\cC_{/X}$.  Write $\Fun_c(X)$ for the group of constructible functions on $X$.
\item 
A \emph{conic Lagrangian cycle} on $T^* X$ is a section of the cosheaf $\cL_c \circ \fj_X$ on $\cC_{/X}$.  Write $\cL_c(X)$ for the group of conic Lagrangian cycles on $T^*X$.
\end{enumerate}
\end{definition}

\begin{example}
\label{ex:quotient}
Let $M$ be a real analytic manifold equipped with a subanalytic action of a finite group $G$, and let $X$ be the quotient stack $M/G$.  By definition $(M \to X)$ is an object of $\cC_{/X}$ whose automorphism group is $G$.  Moreover, every object of $\cC_{/X}$ admits a map to $(M \to X)$.  It follows that $\Fun_c(X) := \varinjlim_{\cC_{/X}} \Fun_c \circ \fj_X$ is naturally identified with the smaller limit $\varinjlim_G \Fun_c(M)$, where the limit is over the one-object diagram with automorphism group $G$.  In other words, $\Fun_c(X)$ is naturally identified with the coinvariants of the $G$-module $\Fun_c(M)$.  An identical argument identifies $\cL_c(X)$ with the $G$-coinvariants of $\cL_c(M)$.
\end{example}

Let us also discuss the cosheaf $\bD$.  For simplicity let us assume $X$ is connected.  We define the constant cosheaf $\bD_X$ on the subanalytic site $\cC_X$ of $X$ by $\bD_X = \bD \circ \fj_X$.  The global sections of $\bD_X$ on $\cC_X$ is naturally given by $H_0(X;\bQ) = \bQ$.  The cosheaf-level maps $\int:\Fun_c \to \bD$ and $\cap \zeta:\cL_c \to \bD$ induce maps $\int:\Fun_c(X) \to \bD_X(X) = \bQ$ and $\cap \zeta:\cL_c(X) \to \bD(X) = \bQ$.  Similarly the cosheaf map $\CC:\Fun_c \to \cL_c$ induces a map $\Fun_c(X) \to \cL_c(X)$.  The following is an immediate consequence of Theorem \ref{thm:localbdk}:

\begin{theorem}
\label{thm:bdkorb}
Let $X$ be a real analytic orbifold.  We have a commutative diagram
$$\xymatrix{
\Fun_c(X) \ar[rr]^{\CC} \ar[dr]_{\int} &  &  \cL_c(X) \ar[dl]^{\cap \zeta} \\ 
& \bQ
}
$$
\end{theorem}

We will give a more concrete description of the groups $\Fun_c(X)$ and $\cL_c(X)$, and of the maps in the triangle of Theorem \ref{thm:bdkorb}, in Section \ref{sec:coarserole}.

\subsection{The role of the coarse space}
\label{sec:coarserole}

Throughout this section, let $X$ denote a real analytic orbifold.  We also consider the orbifold cotangent bundle $T^*X \to X$.  Let $\overline{X}$ denote the coarse space of $X$ and $\overline{T^* X}$ denote the coarse space of $T^* X$.  Let $p:X \to \overline{X}$ denote the natural projection.

\begin{remark}
$\overline{X}$ can be realized naturally as a subanalytic variety, i.e. as a closed subanalytic set in a real analytic manifold, such that the map $p:X \to \overline{X}$ is subanalytic. To see this recall that there exist subanalytic bump functions of class $C^r$ for any finite $r$ (e.g. \cite[Section 3.2]{NZ}).  It follows by a partition of unity argument that we can separate points of $\overline{X}$ by finitely many subanalytic functions $u:X \to \bR$ of class $C^r$, \emph{i.e.} there is a subanalytic map $X \to \bR^n$ of class $C^r$ that induces an injection $\overline{X} \to \bR^n$.  The image of this map is a subanalytic variety.  This remark applies also to $\overline{T^* X}$.
\end{remark}

\begin{remark}
\label{rem:maketwo}
Locally on $\overline{X}$, $X$ looks like a quotient stack.  More precisely, with the subanalytic structure of (1), $\overline{X}$ admits a cover by subanalytic charts 
$\overline{U}$, so that the fiber product stack $\overline{U} \times_{\overline{X}} X$ is isomorphic to the quotient of a real analytic manifold by a finite group acting subanalytically.
\end{remark}

\subsubsection{Constructible functions on the coarse space}

If $\phi:U \to X$ is an object of $\cC_{/X}$, and $f \in \Fun_c(U)$, write $(p\circ \phi)_!(f)$ for the function
$$(p \circ \phi)_!(f)(x) =  \sum_{u \in (p \circ \phi)^{-1}(x)} f(u)$$
The system of maps $(p \circ \phi)_!:\Fun_c(U) \to \Fun_c(\overline{X})$ assemble to a map from the direct limit $\Fun_c(X) \to \Fun_c(\overline{X})$.  Let us denote this map by $p_!$.

\begin{theorem}
\label{thm:46}
Let $X$ be a real analytic orbifold, let $\overline{X}$ be its coarse space, and let $p$ denote the quotient map $p:X \to \overline{X}$.  The map $p_!:\Fun_c(X) \to \Fun_c(\overline{X})$ defined above is an isomorphism.
\end{theorem}

\begin{proof}
By Remark \ref{rem:maketwo}, we may find a covering sieve $\{\overline{U}_i \to \overline{X}\}$ of $X$ such that each $U_i := \overline{U}_i \times_{\overline{X}} X$ is a quotient stack.  Write $p_i:U_i \to \overline{U}_i$.   By the cosheaf property $\Fun_c(X) \cong \varinjlim \Fun_c(U_i)$ and $\Fun_c(\overline{X}) \cong \varinjlim \Fun_c(\overline{U}_i)$, it suffices to prove that the maps $p_{i!}:\Fun_c(U_i) \to \Fun_c(\overline{U}_i)$ are isomorphisms.  In other words, we may reduce to the case when $X$ is a quotient stack.

Thus suppose $X = M/G$ and $\overline{X} = M /\!/G$.  By Example \ref{ex:quotient}, we may identify $\Fun_c(X)$ with the coinvariants $\Fun_c(M)_G$.  On the other hand, pulling back a constructible function on $\overline{X}$ to a constructible function on $M$ identifies $\Fun_c(\overline{X})$ with the $G$-invariants of $\Fun_c(M)$.  

Write $q$ for the map $M \to X$.  For each $x \in \overline{X}$, the fiber $(p \circ q)^{-1}(x) \subset M$ is a $G$-orbit.  Thus, the composite $(p \circ q)_!:\Fun_c(M) \to \Fun_c(M)_G \to \Fun_c(\overline{X})$ takes $h$ to $\sum_{g \in G} h(g^{-1}-)$.  In other words, the map $\Fun_c(M)_G \to \Fun_c(\overline{X}) \cong \Fun_c(M)^G$ is the norm map.  The norm map $V_G \to V^G$ is an isomorphism for every rational vector space $V$ with a $G$-action, so the Theorem follows.
\end{proof}

\subsubsection{Integration and the coarse space}

For simplicity let us assume $X$ is connected.  We define the constant cosheaf $\bD_X$ on the subanalytic site $\cC_X$ of $X$ by $\bD_X = \bD \circ \fj_X$.  The global sections of $\bD_X$ on $\cC_X$ is naturally given by $H_0(X;\bQ) = \bQ$.  The integration map $\int:\Fun_c \to \bD$ between cosheaves on $\cC$ induces a map $\Fun_c(X) \to \bD_X(X) = \bQ$, which we also denote by $\int$.  

\begin{example}
Let $M$, $G$, and $X$ be as in Example \ref{ex:quotient}.  If $h \in \Fun_c(M)$ is of the form $h_1 - g^* h_1$, then $\int_M h = 0$.  It follows that $\int$ descends to an operator on coinvariants $\Fun_c(M)_G \to \bQ$.  If $f \in \Fun_c(X) \cong \Fun_c(M)_G$, and $h$ is any lift of $f$ to $\Fun_c(M)$, then we have $\int_X f = \int_M h$.
\end{example}

We wish to describe this operator in terms of the identification $\Fun_c(X) = \Fun_c(\overline{X})$ of Theorem \ref{thm:46}.

Let us first define a constructible function $\iota:\overline{X} \to \bQ$ by
$$\iota(x) = 1/|G_x| \quad \text{ where $G_x$ is the inertia group of $X$ at $x$}$$

\begin{proposition}
Let $p_!:\Fun_c(X) \to \Fun_c(\overline{X})$ be the identification of Theorem \ref{thm:46}.  Let $f \in \Fun_c(X)$.  Then $\int f = \int p_!(f) \cdot \iota$
\end{proposition}

\begin{proof}
As in the proof of Theorem \ref{thm:46} we can reduce to the case when $X$ is a quotient stack, say $X = M/G$.  For each conjugacy class $C$ of subgroups of $G$, call a function $h \in \Fun_c(M)$ ``$C$-simple'' if it is the indicator function of a locally closed and relatively compact set $K$ whose isotropy groups are all in $C$.  Call an element of $\Fun_c(M)$ ``simple'' if it is $C$-simple for some $C$.  And call an element of $\Fun_c(M)_G$ ``simple'' if it is represented by a simple function in $\Fun_c(M)$.  We can write any $f \in \Fun_c(M)_G$ as a linear combination of simple functions, so it suffices to verify the Proposition in case $f$ itself is simple.  

If $f \in \Fun_c(M)_G$ is represented by the $C$-simple indicator function of a set $K \subset M$, then $\int_X f = \chi_c(K)$.  On the other hand if $m$ is the order of a group in the conjugacy class $C$ then $p_!(f)(x)$ is given by
$$
p_!(f)(x) = \bigg\{
\begin{array}{ll}
m & \text{if $x \in p(K)$} \\
0 & \text{otherwise}
\end{array}
$$
The Proposition follows.
\end{proof}

\subsubsection{Lagrangian cycles in $\overline{T^* X}$}

We can analyze $\cL_c(X)$ in terms of the coarse space $\overline{T^* X}$ of the orbifold $T^* X$.  

\begin{remark}
If $X_1$ and $X_2$ are two orbifolds with $\overline{X_1} \cong \overline{X_2}$, we may still have $\overline{T^*X}_1 \ncong \overline{T^* X}_2$.
\end{remark}

Let $U \to X$ be a map that defines an object of $\cC_{/X}$.  By definition, $U \to X$ is \'etale, and therefore induces a map $T^* U \to \overline{T^*X}$.  Let us say that a subset of $\overline{T^* X}$ is a \emph{coarse Whitney Lagrangian} if its inverse image under all such maps $T^* U \to \overline{T^* X}$ is of the form $\Lambda_S$ in the sense of Section \ref{subsec:lcai}.  If $\Lambda_1$ and $\Lambda_2$ are coarse Whitney Lagrangians with $\Lambda_1 \subset \Lambda_2$, then we have an inclusion $H_n^\BM(\Lambda_1) \to H_n^\BM(\Lambda_2)$, and we will set 
$$W(X) = \bigcup H_n^\BM(\Lambda)$$
where the union runs over coarse Whitney Lagrangians in $\overline{T^* X}$.   Let $W_c(X) \subset W(X)$ denote the subgroup of cycles whose support projects onto a compact set in $\overline{X}$.

\begin{example}
Let $X$ be the quotient orbifold $\bC/\{\pm 1\}$.  Then $\overline{T^*X}$ can be identified with the singular space $\{(u,v,w) \mid uv = w^2\} \subset \bC^3$, where $u$ is the square of the base coordinate and $v$ is the square of the fiber coordinate on $T^* \bC$.  The subset cut out by $u = 0$ belongs to $W_c(X)$.
\end{example}

Let us show that we can identify $\cL_c(X)$ with $W_c(X)$.  For each object $U \to X$ of $\cC_{/X}$, if $S$ is a stratification of $U$ then we may find a Whitney Lagrangian $\Lambda \subset \overline{T^* X}$ such that the map $T^* U \to \overline{T^* X}$ carries $\Lambda_S$ to a subset of $\Lambda$.  For instance, we may take $\Lambda = \Lambda_{S'}$ where $S'$ is a stratification of $X$ whose pullback to $U$ refines $S$.  If $\eta \in H_n^\BM(\Lambda_S)$ has compact horizontal support then its image gives a well-defined element in $H_n^\BM(\Lambda)$, again with compact horizontal support.  This gives us a map $\cL_c(U) \to W_c(X)$, and these assemble to a map from the direct limit $w_X:\cL_c(X) \to W_c(X)$.

\begin{theorem}
\label{thm:LW}
For any real analytic orbifold $X$, the map $w_X:\cL_c(X) \to W_c(X)$ is an isomorphism.
\end{theorem}

\begin{proof}

If $\overline{U}$ is an open subset of the coarse space and $U$ is the corresponding open substack of $X$, we may consider the map $w_U:\cL_c(U) \to W_c(U)$.  The maps $w_U$ and $w_X$ commute with the extension-by-zero maps $\cL_c(U) \to \cL_c(X)$ and $W_c(U) \to W_c(X)$.  In other words, we may regard $w$ as a map of cosheaves on $\overline{X}$.  We may therefore reduce to the case when $X$ is a quotient stack $M/G$.  

In that case by Example \ref{ex:quotient} it suffices to show that the map $\cL_c(M) \to W_c(X)$ identifies $W_c(X)$ with the $G$-coinvariants of $\cL_c(M)$.  Let us denote the map $M \to \overline{X}$ by $\pi$.  Each $G$-invariant stratification $S$ of $M$ determines a $G$-invariant conic Lagrangian $\Lambda_S \subset T^*M$ and a coarse Whitney Lagrangian $\overline{\Lambda}_S \subset \overline{T^* X}$.  Since $G$-invariant stratifications are cofinal among all stratifications of $M$, we are reduced to showing that
$$H_n^\BM(\Lambda_S) \to H_n^\BM(\overline{\Lambda}_S)$$
identifies the codomain with the $G$-coinvariants of the domain. 

Since $\Lambda_S \to \overline{\Lambda}_S$ is a finite map, we may construct a ``transfer'' map $H_n^\BM(\overline{\Lambda}_S) \to H^\BM_n(\Lambda_S)$, that takes $\overline{Z}$ to $\sum_{Z' \in \pi^{-1}(Z)} Z'$.  We complete the proof by noting the composite $H_n^\BM(\Lambda_S) \to H_n^\BM(\overline{\Lambda}_S) \to H_n^\BM(\Lambda_S)$ coincides with the norm map.   
 
\end{proof}

\section{Refinement of Behrend's Theorem}\label{refinedbehrend}

We explain how to use the results here to remove some of the hypotheses in Behrend's work.  We will give a brief review of his definitions, but refer to \cite{Behrend} for a more complete picture.  Note that what we have been calling $\int_X f$, Behrend denotes by $\chi(X,f)$.

Let us first make some remarks about algebraically constructible functions.  If $U$ is a complex algebraic variety, let $\Fun^\alg(U)$ denote the group of $\bQ$-valued functions that are constructible with respect to a complex algebraic stratification.  If $Y$ is a complex algebraic orbifold let $Y_e$ denote the category of \'etale maps $U \to Y$ where $U$ is representable, with its natural Grothendieck topology.  $\Fun^\alg$ forms a sheaf on $Y_e$.  Let $\Fun^\alg(Y)$ denote its global sections.  

\begin{proposition}
\label{prop:cac}
The pullback map $\Fun^{\alg}(\overline{Y}) \to \Fun^{\alg}(Y)$ is an isomorphism.
\end{proposition}

\begin{proof}
As in the proof of Theorem \ref{thm:46}, we may reduce to the case where $Y$ is a quotient stack $M/G$, in which case $\Fun^\alg(Y)$ can be naturally identified with the group of $G$-invariant algebraically constructible functions on $Y$.  
\end{proof}

If $Y$ is proper, we may therefore use Theorem \ref{thm:46} to identify $\Fun^\alg(Y)$ with a subgroup of $\Fun_c(Y)$.  In particular, we can construct elements of $\Fun_c(Y)$ by constructing elements of $\Fun^\alg(Y)$ \'etale locally.  

Similar remarks apply to $\cL_c(Y)$.  Write $\cL^\alg(X) \subset \cL(X)$ for the group of conic Lagrangian cycles that are supported on the conormal variety of some algebraic stratification of $U$.  The assignment $U \mapsto \cL^\alg(U)$ is contravariant for \'etale maps and forms a sheaf on $X_e$.  Let $\cL^\alg(X)$ denote the global sections of this sheaf.  Let $W^\alg(X)$ denote the subgroup of $W(X)$ consisting of cycles whose inverse image along every map $U \to X$ of $X_e$ is algebraic, or equivalently whose support (as a subset of the algebraic space $\overline{T^* X}$) is algebraic.  The proof of Proposition \ref{prop:cac} may be repeated to establish

\begin{proposition}
\label{prop:cac2}
The pullback map $W^\alg(X) \to \cL^\alg(X)$ is an isomorphism.
\end{proposition}

Because of this and Theorem \ref{thm:LW}, when $X$ is proper we may identify $\cL^\alg(X)$ with a subspace of $\cL_c(X)$.

We now explain the application to symmetric obstruction theories.
Let $X$ be a quasiprojective Deligne-Mumford stack over $\bC$.  
A perfect obstruction theory on $X$ is a perfect complex $E \in D^b_{\mathrm{coh}}(X)$ in the derived category of coherent sheaves,
equipped with a map to the cotangent complex
$$\phi: E \rightarrow L_{X}$$
such that $E$ has cohomology only in degrees $[-1,0]$ and such that $\phi$ induces a surjection in degree $-1$ and an isomorphism in
degree $0$.  The obstruction theory is symmetric if, in addition, there exists
an isomorphism
$$\theta: E \rightarrow E^{\vee}[1]$$
such that $\theta^{\vee}[1] = \theta$.
Symmetric obstruction theories arise naturally when studying moduli spaces of stable sheaves on Calabi-Yau threefolds.  The symmetric structure is induced from Serre duality.  

If $X$ carries a symmetric perfect obstruction theory, it also has a virtual fundamental class in degree $0$
$$[X]^{\mathrm{vir}} \in A_0(X),$$
constructed using the data of $(E,\phi)$ above.
Here, $A_0(X)$ denotes the Chow group of zero-cycles on $X$,
In the geometric setting of sheaves on Calabi-Yau threefolds, the degrees of these virtual classes for varying choice of numerical invariants
define the Donaldson-Thomas theory of the threefold.

In \cite{Behrend}, for an arbitrary Deligne-Mumford stack $X$, Behrend also defines a constructible function
$$\nu_X \in \Fun^\alg(X)$$
using the intrinsic normal cone of $X$.  
Behrend uses the index formula for manifolds to prove the following theorem, under the additional conditions that $X$ is either smooth, or else a global finite quotient, or else a gerbe over a scheme. 

\begin{theorem}
If $X$ is proper, then
$$\deg [X]^{\mathrm{vir}} = \chi(X, \nu_X).$$
\end{theorem}

\begin{proof}
Choose a closed embedding
$$X \hookrightarrow Y$$
into a smooth and proper Deligne-Mumford stack $Y$.
We may regard $\nu_X$ as an element of $\Fun^\alg(Y)$ via extension-by-zero.  Since $Y$ is proper we may by Proposition \ref{prop:cac} identify $\nu_X$ with an element of $\Fun_c(Y)$, with associated characteristic cycle $\CC(\nu_Y) \in \cL_c(Y)$.  The restriction of $\CC(\nu_Y)$ to a representable \'etale open set $U \to Y$ belongs to $\cL^\alg(Y)$, so by Proposition \ref{prop:cac2} and the properness of $Y$ we have $\CC(\nu_Y) \in \cL^\alg(Y) \subset \cL_c(Y)$.
It follows from \cite[Proposition 4.16]{Behrend} that
$$[X]^{\mathrm{vir}} = 0^{!}\CC(\nu_X) \in A_0(X)$$
Here $0^!$ denotes refined (to the Chow group) intersection with the zero section of $T^*M$.
Using Theorem \ref{thm:bdkorb}, we have
$$ \deg [X]^{\mathrm{vir}}= \deg 0^! \CC(\nu_X) = \zeta_Y \cap \CC(\nu_X) = \chi(X, \nu_X).$$
\end{proof}

\end{document}